\documentclass[a4paper,10pt]{article}
\usepackage[latin1]{inputenc}
\usepackage[T1]{fontenc}
\usepackage{amsmath}
\usepackage{ifpdf}
\usepackage{amsfonts}
\usepackage{amstext}
\usepackage{amsthm}
\usepackage[final]{pdfpages}
\usepackage[all]{xy}
\usepackage{geometry}
\usepackage{graphicx}
\usepackage{epsfig}
\usepackage{subfigure}
\usepackage{slashed}
\usepackage{color}
\usepackage{amssymb}
\usepackage{srcltx}
\newtheorem{theorem}{Theorem}[section]
\newtheorem{lemma}[theorem]{Lemma}
\newtheorem{defi}[theorem]{Definition}

\newtheorem{cor}[theorem]{Corollary}

\DeclareMathOperator{\ind}{ind}

\DeclareMathOperator{\sfl}{sf}

\title{Bifurcation results for critical points of families of functionals}
\author{Alessandro Portaluri and Nils Waterstraat}

\begin{document}
\date{}
\maketitle

\footnotetext[2]{A. Portaluri was supported by the grant PRIN2009 ``Critical Point Theory and Perturbative Methods
for Nonlinear Differential Equations''.}
\footnotetext[3]{N. Waterstraat was supported by a postdoctoral fellowship of
the German Academic Exchange Service (DAAD).}

\begin{abstract}
Recently the first author studied in \cite{AleBif} the bifurcation of
critical points of 
families of functionals on a Hilbert space, which are parametrised by a compact 
and orientable manifold having a non-vanishing first integral cohomology group.
We improve this
result in two directions: topologically and analytically. From the analytical point
of view we  
generalise it to a broader class of functionals; from the
topological point of view  
we allow the parameter space to be a metrisable Banach manifold. Our methods are in particular powerful if the parameter space is simply connected. As an application of our results we consider families of geodesics in (semi-) Riemannian manifolds. 

\vskip1truecm
\centerline{Dedicated to our mentor Jacobo Pejsachowicz}
\end{abstract}

\section{Introduction}
Let $I:=[0,1]$ denote the unit interval and let $H$ be a separable,
infinite-dimensional real 
Hilbert space. Let $f:I\times H\rightarrow\mathbb{R}$ be a $C^2$ function such
that $0\in H$
is a critical point of all $f_\lambda=f(\lambda,\cdot):H\rightarrow\mathbb{R}$,
$\lambda\in I$. 
An instant $\lambda^\ast\in I$ is called a {\em bifurcation point\/} if every 
neighbourhood of $(\lambda^\ast,0)\in I\times H$ contains points of the 
form $(\lambda,u)$, where $u\neq 0$ is a critical point of $f_\lambda$. The
existence of 
bifurcation points can be studied by considering the associated path of Hessians
$L_\lambda$, $\lambda\in I$, 
of the functions $f_\lambda$ at their critical point $0\in H$. In many geometric variational 
problems these bounded selfadjoint operators are Fredholm.\\
As an example, if $L_\lambda=id-\lambda K$, where $K$ is a selfadjoint compact
operator on $H$, then 
every characteristic value (reciprocal of a non-zero eigenvalue) of $K$ is a
bifurcation point 
according to a classical result of Krasnosel'skii (cf. \cite{Krasnoselskii}).
This result has been 
improved over the years and a bifurcation theorem for general one-parameter
families of functionals 
$f:I\times H\rightarrow\mathbb{R}$ as above was obtained by Fitzpatrick,
Pejsachowicz and 
Recht in \cite{SFLPejsachowicz}. They considered the spectral flow, which is an
integer-valued homotopy
 invariant for paths of bounded selfadjoint Fredholm operators introduced by Atiyah, Patodi and 
Singer in \cite{AtiyahPatodi}. Roughly speaking, the spectral flow of a path
$L$ of selfadjoint
Fredholm operators is the number of negative eigenvalues of $L_0$ that become
positive as the parameter $\lambda$ travels from $0$ to $1$ minus the 
number of positive eigenvalues of $L_0$ that become negative; i.e. the net number of 
eigenvalues which cross zero. 
The main result of \cite{SFLPejsachowicz} shows the existence
of a bifurcation point $\lambda^\ast\in I$ for $f$, if the operators $L_0,L_1$
are invertible and the spectral flow of the path of Hessians $L_\lambda$,
$\lambda\in I$, does not vanish (cf. theorem \ref{Bif} below). Applications of
their result were obtained for bifurcation of periodic orbits of
Hamiltonian systems (cf. \cite{JacoboSFLII}), bifurcation of families of
geodesics in semi-Riemannian manifolds (cf. \cite{MussoPejsachowicz}) and for the 
study of conjugate points in \cite{PicPorTau} and \cite{PicPor}. \\
Recently the first author has considered the more general situation of
$C^2$ functions $f:X\times H\rightarrow\mathbb{R}$, where $X$ is a compact,
orientable smooth manifold of dimension at least $2$ and $0\in H$ is again a
critical point of all functionals $f_\lambda:H\rightarrow\mathbb{R}$,
$\lambda\in X$. In this case the associated family of Hessians $L_\lambda$,
$\lambda\in X$, of $f_\lambda$ at $0\in H$ is a family of bounded selfadjoint
operators on $H$ parametrised by the compact space $X$. If these
operators are in addition Fredholm and strongly indefinite (cf. section 2
below), then one can assign an odd $K$-theory class $\ind_s L\in K^{-1}(X)$,
which was defined by Atiyah, Patodi and Singer in \cite{AtiyahPatodi} and which
can be interpreted as a generalisation of the spectral flow to families. The
main theorem of \cite{AleBif} states that if there exists $\lambda_0\in X$ such
that $L_{\lambda_0}$ is invertible and, moreover, the first Chern class
$c_1(\ind_s L)\in H^1(X;\mathbb{Z})$ is non-trivial, then 

\begin{enumerate}
	\item[i)] $\dim B(f)\geq n-1$, where $\dim$ denotes the Lebesgue
covering dimension,
	\item[ii)] either $B(f)$ disconnects $X$ or it is not contractible to a
point in $X$. 
\end{enumerate} 
The proof of this result uses the bifurcation theorem \cite{SFLPejsachowicz},
\v{C}ech cohomology and Poincaré duality. Note that it is in particular
required that $H^1(X;\mathbb{Z})\neq 0$ and, accordingly, simply connected
parameter spaces like $S^n$, $n\geq 2$, are excluded a priori. In the final
section of \cite{AleBif} the bifurcation theorem is applied to families of
geodesics in semi-Riemannian manifolds. However, since the operators $L$ are
assumed to be strongly indefinite, geodesics in Riemannian manifolds cannot
be treated in this way.\\
The aim of this paper is to improve the results of \cite{AleBif} to more general functionals $f$ and parameter spaces $X$. Our idea is based on the observation that the bifurcation
theorem of \cite{AleBif} can be restated in terms of the spectral flow instead
of $K$-theory classes. Indeed, it is easily seen that the assumptions on the
family $L$ in \cite{AleBif} are equivalent to the existence of a closed path
$\gamma:S^1\rightarrow X$, such that $\gamma(1)$ is invertible and the spectral
flow of $L\circ\gamma$ is non-trivial.\\
Necessarily, our methods are completely different from the ones in \cite{AleBif}.
We assume that $X$ is a metrisable, connected and smooth Banach manifold, but we do not require it to be orientable nor do we make any restrictions on the cohomology
groups of $X$. For each path $\gamma:I\rightarrow X$ in $X$, the composite $L\circ\gamma$ defines a path of selfadjoint Fredholm operators acting on $H$.
If, moreover, $L_{\gamma(0)}$ and $L_{\gamma(1)}$ are invertible, then the
spectral flow $\sfl(L\circ\gamma)$ is defined and assigns to each such path an
integer. Our first theorem states that if we can generate a non-zero integer in
this way, then $B(f)$ either disconnects some open connected subset of $X$ or it has interior points, which is
equivalent to the above i) if $X$ is of finite dimension $n$. Note that, since
we make no assumptions on the cohomology of $X$, $X$ can be contractible and
hence the statement ii) above is obviously wrong in general. However, we
show in a second theorem, whose argument is based on \cite{DimDirac}, that
$B(f)$ disconnects $X$ if $X$ is simply connected and there exists a path
$\gamma:I\rightarrow X$ as above such that $\sfl(L\circ\gamma)\neq 0$. Moreover,
we conclude that $X\setminus B(f)$ has infinitely many path components if we can
find an infinite amount of paths $\gamma$ such that the spectral flows of
$L\circ\gamma$ are pairwise distinct.\\
An important special case, which cannot be treated by the methods in \cite{AleBif}, emerges if the Morse indices of the operators $L$ are finite. Then $\sfl(L\circ \gamma)$ is the difference of the Morse indices of $L_{\gamma(0)}$ and $L_{\gamma(1)}$, so that our invariant does no longer depend on the whole path $\gamma$ but is actually defined for ordered pairs of points in $X$. We want to point out that Morse indices have been computed for the Hessians of many different types of functionals and it is not possible to give an exhaustive list of references here. However, we mention in passing that current work is underway regarding surface theory (cf. \cite{Nayatani}, \cite{Rossman}), the Fermi-Pasta-Ulam problem (cf.\cite{SerraTerracini}), Stokes waves (cf. \cite[\S 11.3]{BuffoniToland}) and semilinear elliptic equations (cf. \cite{Gladiali}, \cite{Dancer}), where in particular families of functionals appear that have infinitely many different Morse indices. Following \cite{AleBif}, we restrict in this paper to applications to families of geodesics in semi-Riemannian manifolds.\\
The paper is structured as follows: in the following section we recall the
definition of the spectral flow, following the approach of
\cite{SFLPejsachowicz}, and we introduce the main theorem of
\cite{SFLPejsachowicz} on the bifurcation of critical points of paths of
functionals. In the third section we state our theorems and discuss some
immediate examples. The fourth section is devoted to the proof of our theorems.
In the final section we consider families of geodesics in semi-Riemannian
manifolds, where we can also treat the case of Riemannian manifolds as well
as non-orientable and non-compact parameter spaces, in contrast to \cite{AleBif}.
Here we assume that the parameter space $X$ is of finite dimension for technical reasons.

%%%%%%%%%%%%%%%%%%%%%%%%%%%%%%%%%%%%%%%%%%%%%%%%%%%%%%%%%%%%%%%%%%%%%%%%%%%%%%%%
%%%%%%%%%%%%%%%%%%%%%%%%%%%%%%%%%%%%%%%%%%%%%%%%%%%%%%%%%%%%%%%%%%%%%%%%%%%%%%%%
%%%%%%%%%%%%%%%%%%%%%%%%%%%%%%%%%%%%%%%%%%%%%%%%%%%%%%%%%%%%%%%%%%%%%%%%%%%%%%%%
%%%%%%%%%%%%%%%%%%%%%%%%%%%%%%%%%%%%%%%%%%%%%%%%%%%%%%%%%%%%%%%%%%%%%%%%%%%%%%%%
%%%%%%%%%%%%%%%%%%%%%%%%%%%%%%%%%%%%%%%%%%

\section{Preliminaries: spectral flow and bifurcation of critical points}

Let  $H$ be a real, separable infinite-dimensional Hilbert space. We denote by $\mathcal{L}(H)$ 
the space of all bounded linear operators acting on $H$ endowed
with the topology induced by the operator norm, and by
$\Phi_S(H)\subset\mathcal{L}(H)$ the subspace of all selfadjoint Fredholm
operators. Recall that $\Phi_S(H)$ consists of three connected components. Two
of them are given by

\begin{align*}
\Phi^+_S(H)=\{L\in \Phi_S(H):\sigma_{ess}(L)\subset(0,\infty)\},\qquad
\Phi^-_S(H)=\{L\in \Phi_S(H):\sigma_{ess}(L)\subset(-\infty,0)\}
\end{align*}
and their elements are called essentially positive or essentially negative,
respectively. Both of them are contractible as topological spaces. Elements of
the remaining component $\Phi^i_S(H)=\Phi_S(H)\setminus(\Phi^+_S(H)\cup
\Phi^-_S(H))$ are called strongly indefinite. $\Phi^i_S(H)$ has the same
homotopy groups as the stable orthogonal group and the \textit{spectral flow},
which we now want to introduce briefly, induces an isomorphism between its
fundamental group and the integers. We follow the approach developed by
Fitzpatrick, Pejsachowicz and Recht in \cite{SFLPejsachowicz}.\\
If $S,T$ are two selfadjoint invertible operators such that $S-T$ is compact,
then the difference of their spectral projections with respect to any spectral
subset is compact as well. Hence, denoting by $E_-(\cdot)$, $E_+(\cdot)$ the
negative and positive subspaces of a selfadjoint operator, the \textit{relative
Morse index}

\begin{align*}
\mu_{rel}(S,T)=\dim(E_-(S)\cap E_+(T))-\dim(E_+(S)\cap E_-(T))
\end{align*}
is well defined and finite.\\
The group $GL(H)$ of all invertible operators on $H$ acts on $\Phi_S(H)$ by
\textit{cogredience}, sending $M\in GL(H)$ to $M^\ast LM$, $L\in\Phi_S(H)$. This
induces an action of paths in $GL(H)$ on paths in $\Phi_S(H)$. One of the main
theorems in \cite{SFLPejsachowicz} states that for any path
$L:I\rightarrow\Phi_S(H)$ there exist a path $M:I\rightarrow GL(H)$ and an invertible operator $J\in\Phi_S(H)$, such that $M^\ast_tL_tM_t=J+K_t$ with $K_t$ compact for each $t\in[0,1]$.

\begin{defi}
Let $L:I\rightarrow\Phi_S(H)$ be a path such that $L_0$ and $L_1$ are invertible. Then
the \textit{spectral flow} of $L$ is the integer

\begin{align*}
\sfl(L)=\mu_{rel}(J+K_0,J+K_1),
\end{align*}
where $\{J+K_t\}_{t \in I}$ is any path of compact perturbations of an invertible operator $J\in\Phi_S(H)$ which is cogredient with
$\{L_t\}_{t \in I}$ in the sense above.
\end{defi}
It follows from general properties of the relative Morse index that the spectral
flow does not depend on the choices of $J$ and $K$. Moreover, it is obviously
preserved by cogredience.\\
The main properties of the spectral flow are:

\begin{itemize}
	\item if $L_{t}$ is invertible for all $t\in I$, then $\sfl L=0$.
	\item If $L^1$, $L^2$ have invertible ends and the concatenation
$L^1\ast L^2$ is defined, then $\sfl(L^1\ast L^2)=\sfl(L^1)+\sfl(L^2)$.
	\item Let $h:I\times I\rightarrow\Phi_S(H)$ be a homotopy such that
$h(s,0)$ and $h(s,1)$ are invertible for all $s\in I$. Then
	\begin{align*}
	\sfl(h(0,\cdot))=\sfl(h(1,\cdot)).
	\end{align*}
  \item If $L_t\in\Phi^+(H)$, $t\in I$, and $L_0, L_1$ are invertible, then the spectral flow of $L$ is the difference of the Morse indices at its endpoints:
	
	\begin{align*}
	\sfl(L)=\mu_{Morse}(L_0)-\mu_{Morse}(L_1).
	\end{align*}
	\item If $L$ has invertible ends and $\tilde{L}_t=L_{1-t}$, $t\in I$, then $\sfl\tilde{L}=-\sfl L$.
\end{itemize}
Finally, we introduce the main result of \cite{SFLPejsachowicz} on bifurcation
of critical points of paths of functionals. 

\begin{theorem}\label{Bif}
Let $f:I\times H\rightarrow\mathbb{R}$ be a $C^2$ function such that for each
$\lambda\in I$, $0$ is a critical point of the functional
$f_\lambda=f(\lambda,\cdot)$. Assume that the Hessians $L_\lambda$ of
$f_\lambda$ at $0$ are Fredholm operators and that $L_0$ and $L_1$ are
invertible. If $\sfl L\neq 0$, then every neighbourhood of $I\times\{0\}$ in
$I\times H$ contains points of the form $(\lambda,u)$, where $u\neq 0$ is a
critical point of $f_\lambda$.
\end{theorem}

As already mentioned in the introduction, the aim of this paper is to generalise
theorem \ref{Bif} to families of functionals which are parametrised by Banach
manifolds instead of the unit interval.

\section{The theorems}
Let $H$ be a Hilbert space and $X$ a metrisable, connected and smooth Banach
manifold (cf. \cite[II]{Lang}). Let $f:X\times H\rightarrow\mathbb{R}$ be a
$C^2$ function such that $0\in H$ is a critical point of all functionals
$f_\lambda=f(\lambda,\cdot):H\rightarrow\mathbb{R}$, $\lambda\in X$. We refer to
$X\times\{0\}$ as the trivial branch of critical points of the family
$\{f_\lambda\}_{\lambda\in X}$ and we denote henceforth by $L_\lambda$ the
associated Hessian of $f_\lambda$ at $0$, $\lambda\in X$. Note that each
$L_\lambda$ is a bounded selfadjoint operator on the Hilbert space $H$ and the
family $L:X\rightarrow\mathcal{L}(H)$ depends continuously on $\lambda\in X$. In
what follows we assume that all $L_\lambda$, $\lambda\in X$, are Fredholm
operators.\\
We call a path $\gamma:I\rightarrow X$ admissible if $L_{\gamma(0)}$ and
$L_{\gamma(1)}$ are invertible. In this case $L_{\gamma(t)}$, $t\in I$, defines
a path in $\Phi_S(H)$ having invertible endpoints and, accordingly, the spectral
flow of $L\circ\gamma$ is well defined, which we will denote henceforth by
$\sfl(\gamma,L)$.\\  
We call $\lambda^\ast\in X$ a \textit{bifurcation point of critical points} from
the trivial branch if there exist a sequence $\lambda_n\rightarrow\lambda^\ast$
in $X$ and a sequence $u_n\rightarrow 0$ in $H$ such that $u_n$ is a non-zero critical point of $f_{\lambda_n}$, $n\in\mathbb{N}$. We denote by $B(f)\subset
X$ the set of all bifurcation points and observe that

\begin{align}\label{singularset}
B(f)\subset\Sigma(L):=\{\lambda\in X:\ker L_\lambda\neq \{0\}\}
\end{align}
by the inverse function theorem. It is a direct consequence of the definition that $B(f)$ is closed in $X$. 
With all this in place, our first theorem reads as follows:

\begin{theorem}\label{theoremI}
Let $X$ admit smooth partitions of unity. If there exists an admissible path $\gamma$
in $X$ such that $\sfl(\gamma,L)\neq 0$, then either $B(f)$ has interior points
or it disconnects some open connected subset of $X$. 
\end{theorem}
Note that by inclusion \eqref{singularset} we can possibly exclude one of these
alternatives if $\Sigma(L)$ is known. Moreover, observe that every paracompact smooth Hilbert manifold admits smooth partitions of unity (cf. \cite[Corollary II.3.8]{Lang}). \\
From elementary dimension theory (cf. e.g. \cite[5,\S 2]{Dimension}) we obtain
the following corollary:

\begin{cor}\label{cor}
If $\dim X=n<\infty$ and the assumptions of theorem \ref{theoremI} hold, then
$\dim B(f)\geq n-1$, where $\dim$ denotes the Lebesgue covering dimension.
\end{cor}

The assertion of the corollary is the main result of the first author's article
\cite{AleBif}, in which it is assumed in addition that $X$ is compact and
orientable, that the Hessians $L_\lambda$, $\lambda\in X$, are strongly
indefinite and that the path $\gamma$ is closed.\\
In case $X$ is simply connected, theorem \ref{theoremI} can be
improved like this:

\begin{theorem}\label{theoremII}
Let $X$ be simply connected.
\begin{itemize}
	\item[i)] If there exists an admissible path $\gamma$ such that
$\sfl(\gamma,L)\neq 0$, then $B(f)$ disconnects $X$.
	\item[ii)] If there exists a sequence of admissible paths $\{\gamma_k\}_{k \in \mathbb{N}}$,
such that 

\begin{align*}
\lim_{k\to+\infty}|\sfl(\gamma_k,L)|=+\infty,
\end{align*}
then $X\setminus B(f)$ has infinitely many path components.   
\end{itemize}
\end{theorem}

In case of a simply connected parameter space $X$, theorem
\ref{theoremII} $i)$ implies theorem \ref{theoremI}. Moreover, note that we do not
require $X$ to admit smooth partitions of unity in theorem \ref{theoremII}.\\
We close this section by providing two immediate examples. The first makes
theorem \ref{theoremII} reminiscent of the classical Krasnosel'skii theorem for
variational bifurcation (cf. \cite{Krasnoselskii}), which we have already
mentioned in the introduction. Let $\mathcal{K}_{s}(H)$ denote the closed
subspace of $\mathcal{L}(H)$ consisting of all compact selfadjoint operators
acting on $H$. Let $g:\mathcal{K}_s(H)\times H\rightarrow\mathbb{R}$ be a $C^2$ function such that $g(K,u)=o(\|u\|^2)$, $u\rightarrow 0$, for all $K\in\mathcal{K}_s(H)$. We consider the family of functionals

\begin{align*}
f:\mathcal{K}_s(H)\times H\rightarrow\mathbb{R},\quad f(K,u)=\frac{1}{2}\langle
(id+K)u,u\rangle_H+g(K,u)
\end{align*}
and note that $0\in H$ is a critical point of each $f_K$. The Hessian of $f_K$
at $0\in H$ is given by $L_K=id+K\in\Phi^+_S(H)$, $K\in\mathcal{K}_s(H)$.\\
Now let $\{e_i\}_{i\in\mathbb{N}}$ be a complete orthonormal system of $H$ and
define a sequence of compact selfadjoint operators by

\begin{align*}
K_nu=-2\sum^n_{i=1}{\langle u,e_i\rangle e_i},\quad n\in\mathbb{N},
\end{align*} 
and a sequence of paths by

\begin{align*}
\gamma_n:I\rightarrow\mathcal{K}_S(H),\quad \gamma_n(t)=tK_n.
\end{align*}
It is clear that each path $\gamma_n$ is admissible with respect to $L$ and we
obtain from the very definition of the spectral flow that

\begin{align*}
\sfl(\gamma_n,L)=\mu_{rel}(id+K_n,id)=\dim(E_-(id+K_n)\cap
E_+(id))-\dim(E_+(id+K_n)\cap E_-(id))=n.
\end{align*}
We conclude from theorem \ref{theoremII} ii) that $\mathcal{K}_s(H)\setminus
B(f)$ has infinitely many path components.\\
Our second example shows that simply connectedness is a necessary assumption
in theorem \ref{theoremII}. It is easy to construct a family of functionals
$f:S^1\times H\rightarrow\mathbb{R}$ such that the associated family of Hessians
$L_\lambda$ is not invertible precisely at $1\in S^1$ and $\sfl L\neq 0$. Hence
$B(f)=\{1\}\subset S^1$ by theorem \ref{Bif}. Now we consider the torus
$T=S^1\times S^1$ and the family of functionals 

\begin{align*}
\tilde{f}:T\times H\rightarrow\mathbb{R},\quad
\tilde{f}(\lambda_1,\lambda_2,u)=f(\lambda_1,u).
\end{align*}
The path $\gamma:S^1\rightarrow T$, $\gamma(\lambda)=(\lambda,1)$, is such that
$\sfl(\gamma,L)\neq 0$ and hence we infer from corollary \ref{cor} that $\dim
B(\tilde{f})\geq 1$, whereas theorem \ref{theoremII} cannot be applied to
$\tilde{f}$ since $T$ is not simply connected. Indeed, $B(\tilde{f})=\{1\}\times
S^1$ which does not disconnect $T$. Note, however, that $B(\tilde{f})$ is not
contractible to a point in $X$, in accordance with \cite{AleBif}.

\section{The proofs}
We fix a compatible metric $\rho$ on the smooth Banach manifold $X$, which we have
assumed to be metrisable, and denote henceforth by $B(\lambda,\varepsilon)$ the
open ball of radius $\varepsilon$ around $\lambda\in X$. The metric $\rho$
induces a metric on the space $C(I,X)$ of all continuous paths in $X$ by
$d(\gamma_1,\gamma_2)=\sup_{t\in I}\rho(\gamma_1(t),\gamma_2(t))$. We will use
several times the following elementary approximation result, whose proof we
leave to the reader.

\begin{lemma}\label{approx}
Let $\gamma_1:I\rightarrow X$ be a continuous path and $\varepsilon>0$. Then
there exists a smooth path $\gamma_2:I\rightarrow X$ such that
$\gamma_2(0)=\gamma_1(0)$, $\gamma_2(1)=\gamma_1(1)$ and
$d(\gamma_1,\gamma_2)<\varepsilon$.
\end{lemma} 
 
We now prove our theorems \ref{theoremI} and \ref{theoremII} in two separate
sections.

\subsection{Proof of theorem \ref{theoremI}}
We assume that $B(f)$ neither has interior points nor disconnects some connected open subset of $X$, and that $\gamma:I\rightarrow X$ is an admissible path such that
$\sfl(\gamma,L)\neq 0$. Our aim is to reach a contradiction.\\
We fix any spray on $X$, which exists since we assume that $X$ admits smooth partitions
of unity (cf. \cite[Theorem IV.3.1]{Lang}). Let $\exp:\mathcal{D}\rightarrow X$
denote the associated exponential map of this spray, where $\mathcal{D}\subset
TX$ denotes its open maximal domain. We obtain a smooth map

\begin{align*}
F:\mathcal{D}\rightarrow X\times X,\quad F(v)=(\pi(v),\exp_{\pi(v)}(v)),
\end{align*}
where $\pi:TX\rightarrow X$ is the projection of the tangent bundle. By
\cite[Prop. VIII.5.1]{Lang}, $F$ is a local diffeomorphism on the zero section of $TX$, and arguing as in \cite[\S 12]{Broecker} we can choose a subset  $\tilde{D}\subset\mathcal{D}$ such
that $\tilde{D}_\lambda=\mathcal{D}\cap T_\lambda X$ is star-shaped with respect
to $0\in T_\lambda X$, $\lambda\in X$, and $F$ is a diffeomorphism from
$\tilde{D}$ onto an open neighbourhood $\mathcal{U}$ of the diagonal in $X\times
X$. Finally, we choose $\varepsilon>0$ small enough so that
$(\tilde{\gamma}(t),\gamma(t))\in\mathcal{U}$, $t\in I$, for any path
$\tilde{\gamma}:I\rightarrow X$ such that
$d(\tilde{\gamma},\gamma)<3\varepsilon$.\\
Our next goal is to construct a path in $X$ which is $2\varepsilon$-close to $\gamma$ but does
not intersect $B(f)$.
\begin{figure}
\centering
{%
\includegraphics[width=0.600\textwidth]{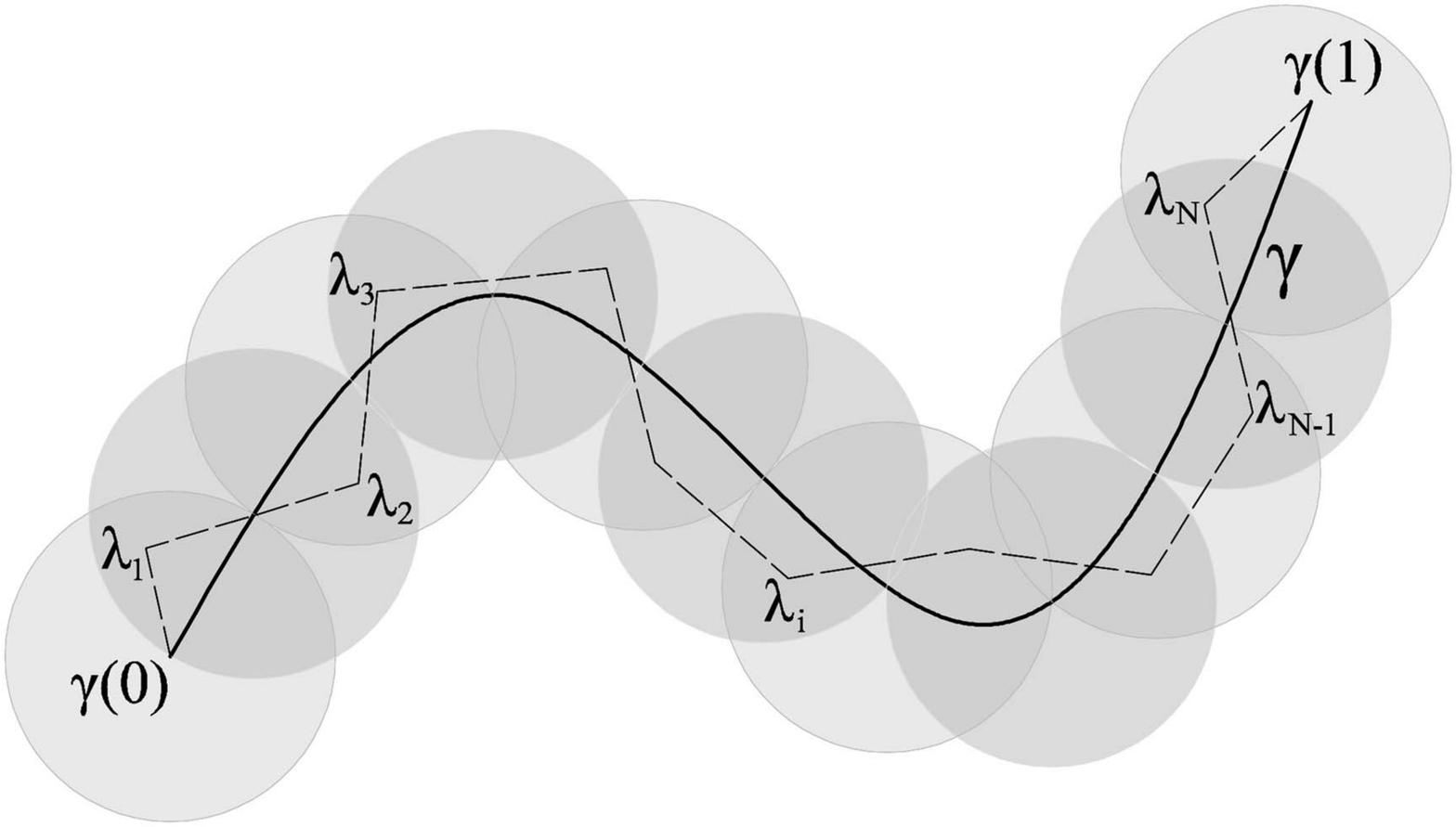}
}
\caption{Construction of  $\gamma_1\ast\ldots\ast\gamma_{N+1}$.}
\label{fig:figura1}
\end{figure}
At first, if $\gamma(I)\subset B(\gamma(0),\varepsilon)$, there exists by assumption a path in $B(\gamma(0),\varepsilon)$ connecting $\gamma(0)$ and $\gamma(1)$ without intersecting $B(f)$. Hence let us assume that $\gamma(I)\not\subset B(\gamma(0),\varepsilon)$. We define 

\begin{align*}
t_1:=\inf\{t\in I: \gamma(t) \in \partial B(\gamma(0), \varepsilon)\}.
\end{align*}
By assumption there exists $\lambda_1\in B(\gamma(0),\varepsilon)\cap
B(\gamma(t_1),\varepsilon)$ such that $\lambda_1\notin B(f)$, and moreover there
exists a path 

\begin{align*}
\gamma_1:[0,t_1]\rightarrow B(\gamma(0),\varepsilon)
\end{align*}
joining $\gamma(0)$ and $\lambda_1$ without intersecting $B(f)$. Note that 

\begin{align*}
\sup_{t\in[0,t_1]}\rho(\gamma_1(t),\gamma(t))<2\varepsilon.
\end{align*}
If $\gamma([t_1,1])\subset B(\gamma(t_1),\varepsilon)$, then we can again finish the construction by using the assumption that $B(f)$ does not disconnect $B(\gamma(t_1),\varepsilon)$. Otherwise, we define 

\begin{align*}
t_2:=\inf\{t_1\leq t\leq 1:\gamma(t)\in\partial B(\gamma(t_1),\varepsilon)\}.
\end{align*}
As before there are $\lambda_2\in B(\gamma(t_1),\varepsilon)\cap
B(\gamma(t_2),\varepsilon)$ such that $\lambda_2\notin B(f)$ and a path 

\begin{align*}
\gamma_2:[t_1,t_2]\rightarrow B(\gamma(t_1),\varepsilon)
\end{align*}
connecting $\lambda_1$ and $\lambda_2$ without intersecting $B(f)$ and such that

\begin{align*}
\sup_{t\in[t_1,t_2]}\rho(\gamma_2(t),\gamma(t))<2\varepsilon.
\end{align*}
If we continue this process, we eventually arrive at an index
$N\in\mathbb{N}$ such that the path 

\begin{align*}
\gamma_N:[t_{N-1},t_N]\rightarrow B(\gamma(t_{N-1}),\varepsilon)
\end{align*}
does not intersect $B(f)$, plus $\gamma_N(t_{N-1})=\gamma_{N-1}(t_{N-1})$,
$\lambda_N:=\gamma_N(t_N)\in B(\gamma(t_N),\varepsilon)$ and 

\begin{align*}
\sup_{t\in[t_{N-1},t_N]}\rho(\gamma_N(t),\gamma(t))<2\varepsilon,
\end{align*}
but

\begin{align*}
\gamma([t_N,1])\cap\partial B(\gamma(t_N),\varepsilon)=\emptyset.
\end{align*}
If this was not true, then we would obtain a sequence $\{t_n\}_{n\in\mathbb{N}}\subset I$, such
that $t_n\rightarrow 1$ but $\gamma(t_n)\nrightarrow\gamma(1)$,
$n\rightarrow\infty$, in contradiction to the continuity of $\gamma$.\\ 
Accordingly, $\gamma(1)\in B(\gamma(t_N),\varepsilon)$ and we can find a path 

\begin{align*}
\gamma_{N+1}:I\rightarrow B(\gamma(t_N),\varepsilon),
\end{align*}
which connects $\lambda_N$ and $\gamma(1)$ without intersecting $B(f)$. Now
the concatenation 

\begin{align*}
\gamma_1\ast\ldots\ast\gamma_{N+1}:I\rightarrow X
\end{align*}
does not intersect $B(f)$ and 

\begin{align*}
d(\gamma_1\ast\ldots\ast\gamma_{N+1},\gamma)<2\varepsilon.
\end{align*}     
Since $B(f)$ is closed, we can use lemma \ref{approx} in order to choose a
smooth path 

\begin{align*}
\tilde{\gamma}:I\rightarrow X\,\, \text{such that}\,\, \tilde{\gamma}(I)\cap
B(f)=\emptyset\,\,\text{and}\,\, d(\tilde{\gamma},\gamma)<3\varepsilon.
\end{align*}
Now we consider the homotopy

\begin{align*}
h:I\times I\rightarrow X,\quad h(s,t)=\exp(s\cdot
(F\mid_{\tilde{D}})^{-1}(\gamma(t),\tilde{\gamma}(t))).
\end{align*}
Note that $h(s,0)=\gamma(0)=\tilde{\gamma}(0)$,
$h(s,1)=\gamma(1)=\tilde{\gamma}(1)$ for all $s\in I$, and
$h(0,\cdot)=\gamma$, $h(1,\cdot)=\tilde{\gamma}$ and hence
$\sfl(\tilde{\gamma},L)=\sfl(\gamma,L)\neq 0$ by the homotopy invariance of the
spectral flow. We define

\begin{align*}
\tilde{f}:I\times H\rightarrow\mathbb{R},\quad
\tilde{f}(t,u)=f(\tilde{\gamma}(t),u),
\end{align*}
which is a $C^2$ function. Each
$\tilde{f}_t=\tilde{f}(t,\cdot):H\rightarrow\mathbb{R}$ has $0\in H$ as critical
point and the associated Hessian is given by $L_{\tilde{\gamma}(t)}$, $t\in I$.
According to theorem \ref{Bif}, there exist a sequence
$\{t_n\}_{n\in\mathbb{N}}\subset I$ and a sequence
$\{u_n\}_{n\in\mathbb{N}}\subset H$ such that $u_n\neq 0$ is a critical point of
$\tilde{f}_{t_n}$, $n\in\mathbb{N}$, and $(t_n,u_n)\rightarrow(t^\ast,0)$ for
some $t^\ast\in I$. Now the sequence
$\{(\tilde{\gamma}(t_n),u_n)\}_{n\in\mathbb{N}}\subset X\times H$ shows that
$\gamma(t^\ast)\in X$ is a bifurcation point of $f$ on $\tilde{\gamma}(I)$,
contradicting the construction of $\tilde{\gamma}$, and this ends the 
proof.

\subsection{Proof of theorem \ref{theoremII}}
At first we prove statement i) arguing by contradiction. 
\begin{figure}[ht]
\centering
{%
\includegraphics[width=0.400\textwidth]{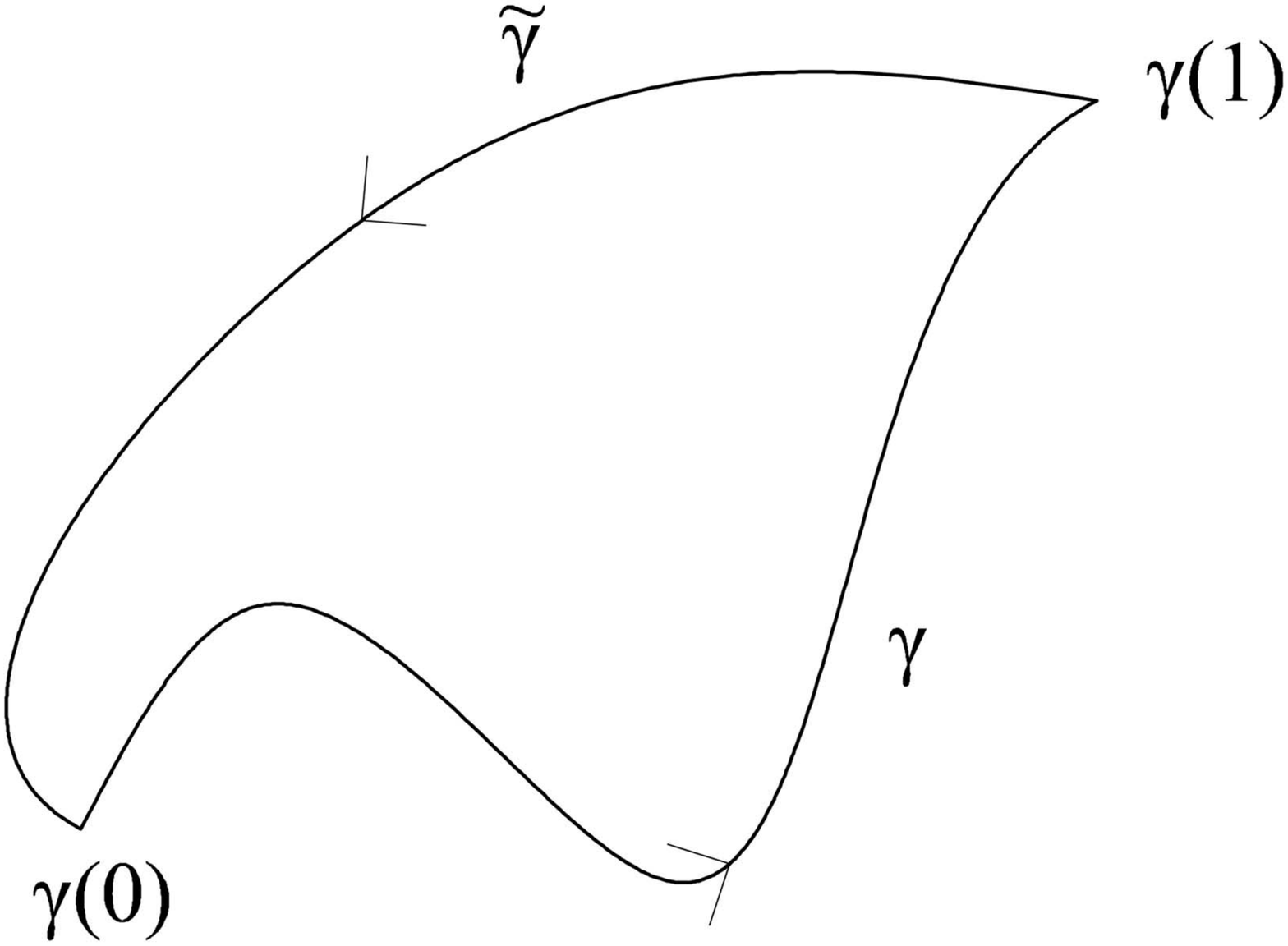}
}
\caption{A closed path with vanishing spectral flow.}
\label{fig:figura2}
\end{figure}
Let $\gamma$ be an admissible path such that
$\sfl(\gamma,L)\neq 0$ and assume that $X\setminus B(f)$ is connected. Then we
can find a path $\tilde{\gamma}$ such that $\tilde{\gamma}(0)=\gamma(1)$,
$\tilde{\gamma}(1)=\gamma(0)$ and $\tilde{\gamma}(I)\cap B(f)=\emptyset$. Since
$B(f)$ is closed we can assume without loss of generality that $\tilde{\gamma}$
is smooth by lemma \ref{approx}. Arguing as in the last paragraph of the
previous section, we conclude from theorem \ref{Bif} that
$\sfl(\tilde{\gamma},L)=0$.

Now we consider the concatenation $\gamma\ast\tilde{\gamma}$ which is a closed
path. Since $X$ is simply connected, $\gamma\ast\tilde{\gamma}$
is homotopic to the constant path $\gamma'(t)\equiv\gamma(0)$, $t\in I$, by means of a homotopy leaving $\gamma(0)$ fixed. Hence

\begin{align*}
\sfl(\gamma,L)=\sfl(\gamma,L)+\sfl(\tilde{\gamma},L)=\sfl(\gamma\ast\tilde{
\gamma},L)=\sfl(\gamma',L)=0
\end{align*}  
which is a contradiction to our assumption.\\
\begin{figure}[ht]
\centering
{%
\includegraphics[width=0.500\textwidth]{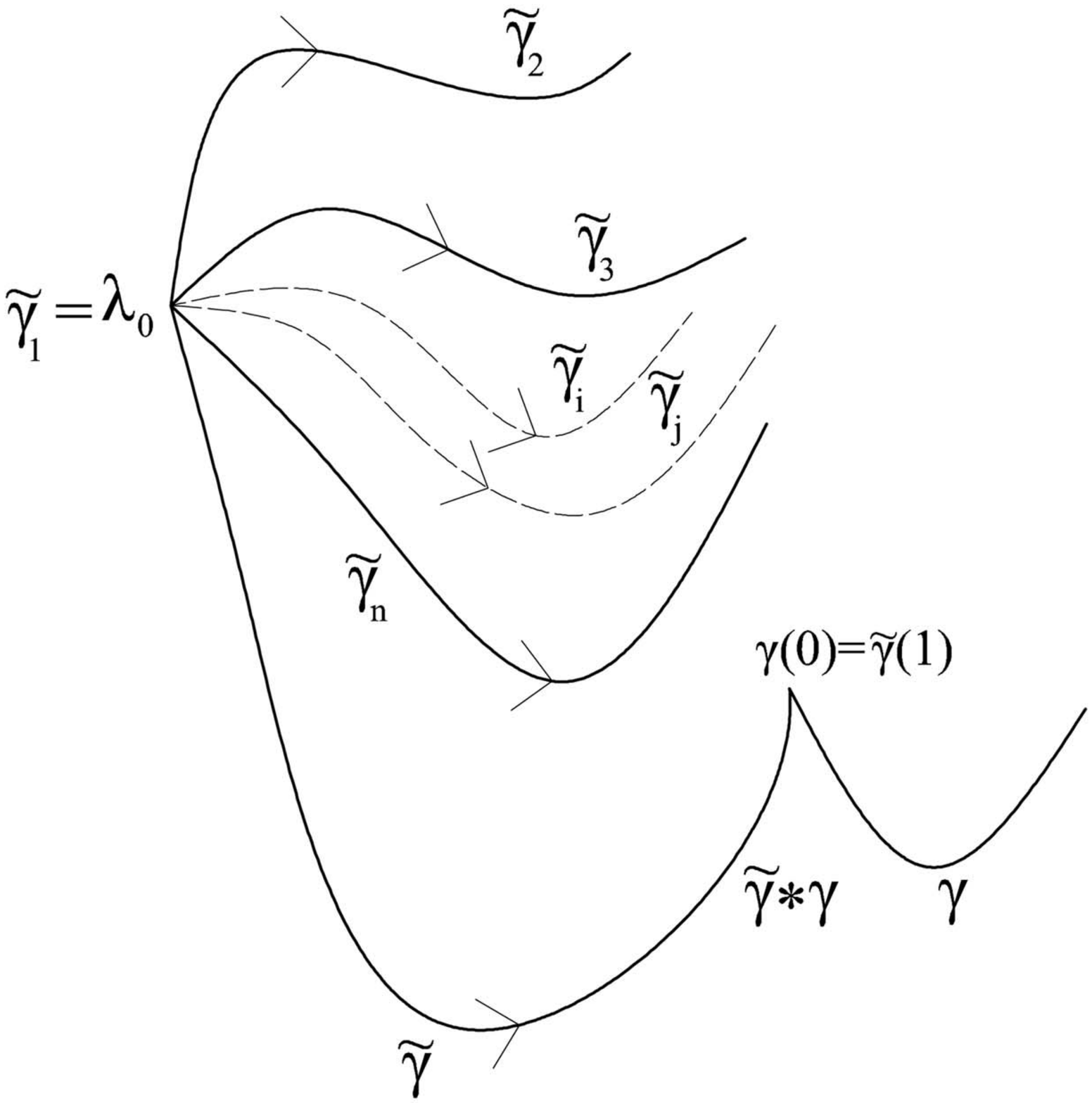}
}
\caption{The inductive construction of $\tilde{\gamma}_k$.}
\label{fig:figura3}
\end{figure}
In order to prove ii) we first construct inductively a sequence
$\tilde{\gamma}_k$, $k\in\mathbb{N}$, of admissible paths in $X$ such that
$\sfl(\tilde{\gamma}_i,L)\neq\sfl(\tilde{\gamma}_j,L)$ for all $i\neq j$ and
$\tilde{\gamma}_k(0)=\lambda_0$, $k\in\mathbb{N}$, for some $\lambda_0\in X$.\\
Let $\lambda_0\in X$ be such that $L_{\lambda_0}$ is invertible. Set
$\tilde{\gamma}_1\equiv \lambda_0$ and assume henceforth that we have already
constructed paths $\tilde{\gamma}_1,\ldots,\tilde{\gamma}_n$ such that
$\sfl(\tilde{\gamma}_i,L)\neq\sfl(\tilde{\gamma}_j,L)$ for $i\neq j$ and
$\tilde{\gamma}_i(0)=\lambda_0$ for all $1\leq i\leq n$.\\
We choose an admissible path $\gamma$ such that
\begin{align*}
|\sfl(\gamma,L)|>\max_{1\leq i,j\leq
n}|\sfl(\tilde{\gamma}_i,L)-\sfl(\tilde{\gamma}_j,L)|
\end{align*} 
and a path $\tilde{\gamma}$ such that $\tilde{\gamma}(0)=\lambda_0$ and
$\tilde{\gamma}(1)=\gamma(0)$. If $\sfl(\tilde{\gamma}\ast\gamma,L)\neq
\sfl(\tilde{\gamma}_i,L)$ for all $1\leq i\leq n$, then we set
$\tilde{\gamma}_{n+1}=\tilde{\gamma}\ast\gamma$. Otherwise we set
$\tilde{\gamma}_{n+1}=\tilde{\gamma}$. In order to justify our choice, assume
that $\sfl(\tilde{\gamma}\ast\gamma,L)=\sfl(\tilde{\gamma}_i,L)$ and 
$\sfl(\tilde{\gamma},L)=\sfl(\tilde{\gamma}_j,L)$ for some $1\leq i,j\leq n$.
Then

\begin{align*}
\sfl(\tilde{\gamma}_i,L)=\sfl(\tilde{\gamma}\ast\gamma,L)=\sfl(\tilde{\gamma},
L)+\sfl(\gamma,L)=\sfl(\tilde{\gamma}_j,L)+\sfl(\gamma,L)
\end{align*}
which is a contradiction to the choice of $\gamma$.\\
Hence there exists a sequence of admissible paths $\tilde{\gamma}_k$,
$k\in\mathbb{N}$, such that
$\sfl(\tilde{\gamma}_i,L)\neq\sfl(\tilde{\gamma}_j,L)$ for all $i\neq j$ and
$\tilde \gamma_k(0)=\lambda_0$, $k\in\mathbb{N}$, for some $\lambda_0\in X$. We claim
that the endpoints $\tilde{\gamma}_i(1)$, $i\in\mathbb{N}$, all lie in different
path components of $X\setminus B(f)$.\\
Assume on the contrary that there exists a path $\gamma'$ connecting
$\tilde{\gamma}_i(1)$ and $\tilde{\gamma}_j(1)$ and such that $\gamma'(I)\cap
B(f)=\emptyset$. Again we can suppose by lemma \ref{approx} that $\gamma'$ is smooth without loss of
generality and we conclude that $\sfl(\gamma',L)=0$ by
theorem \ref{Bif}. The concatenation
$\tilde{\gamma}_i\ast\gamma'\ast\tilde{\gamma}_{-j}$ is a closed path, where we
denote by $\tilde{\gamma}_{-j}(t)=\tilde{\gamma}_j(1-t)$, $t\in I$, the inverse
path. We infer as in the proof of i) above that
$\sfl(\tilde{\gamma}_i\ast\gamma'\ast\tilde{\gamma}_{-j},L)=0$ and obtain
\begin{figure}[ht]
\centering
{%
\includegraphics[width=0.400\textwidth]{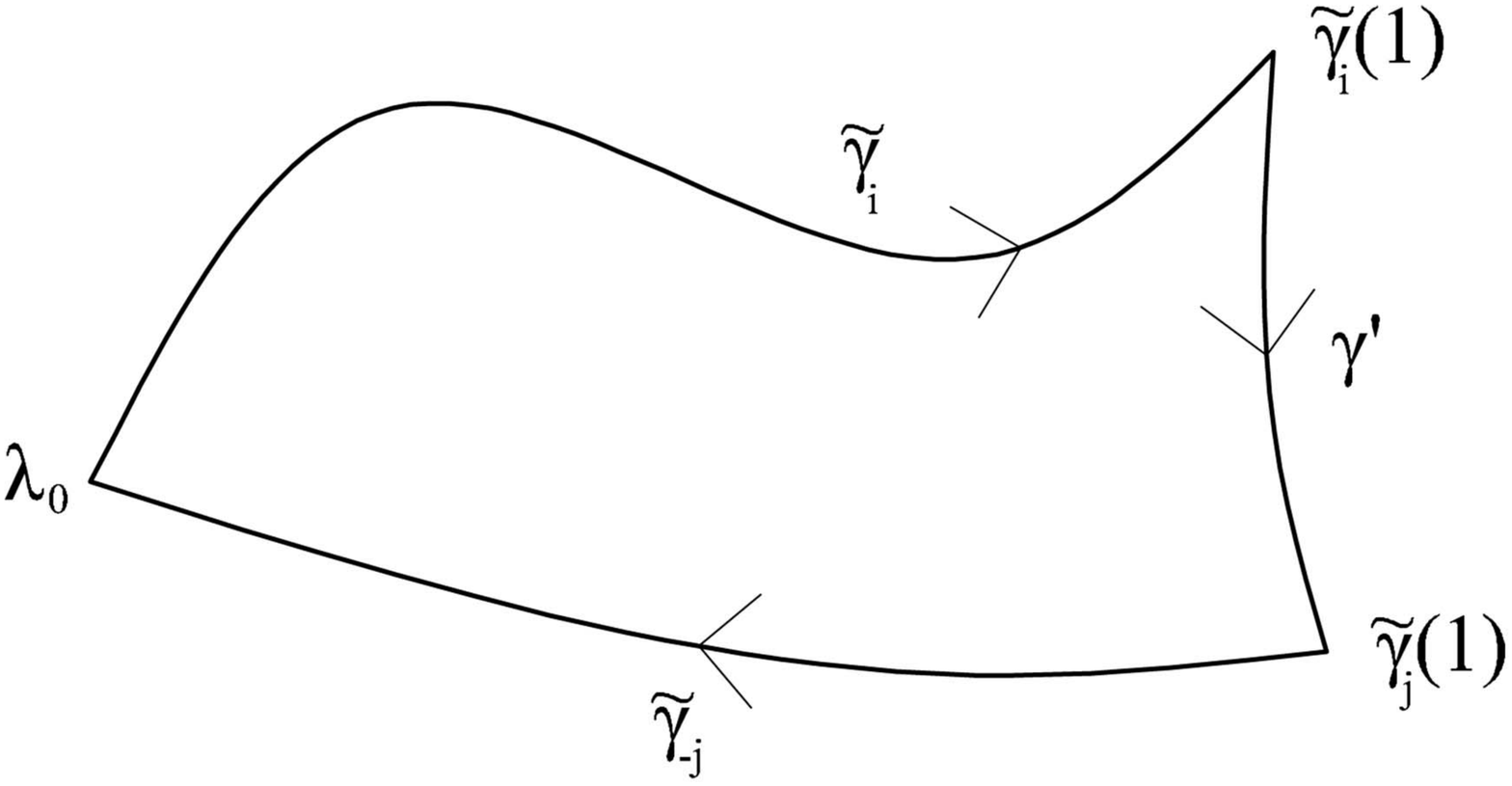}
}
\caption{The construction of $\gamma'$.}
\label{fig:figura4}
\end{figure}
\begin{align*}
0=\sfl(\tilde{\gamma}_i\ast\gamma'\ast\tilde{\gamma}_{-j},L)=\sfl(\tilde{\gamma}
_i,L)+\sfl(\gamma',L)+\sfl(\tilde{\gamma}_{-j},L)=\sfl(\tilde{\gamma}_i,
L)-\sfl(\tilde{\gamma}_j,L),
\end{align*} 
which is a contradiction to the existence of the paths $\tilde{\gamma}_k$,
$k\in\mathbb{N}$. This concludes the proof.

\section{Families of geodesics}
In this final section we apply our theory to families of geodesics in
semi-Riemannian manifolds in order to improve the
corresponding result from \cite[\S 4]{AleBif}. We begin by recalling briefly
some constructions from \cite{Pejsachowicz}. The latter article is, together with \cite{MussoPejsachowicz}, the main reference for a more detailed introduction
to the present setting.\\
Let $M$ be a smooth connected manifold of dimension $n\in\mathbb{N}$. We denote
by $\Omega$ the Hilbert manifold of all $H^1$ curves in $M$, which is modelled
on the Sobolev space $H^1(I,\mathbb{R}^n)$. The tangent space $T_\gamma\Omega$
at an element $\gamma\in\Omega$ can be identified in a natural way with the
space $H^1(\gamma)$ of all $H^1$ vector fields along $\gamma$. The endpoint map

\begin{align*}
\pi:\Omega\rightarrow M\times M,\quad \pi(\gamma)=(\gamma(0),\gamma(1))
\end{align*}
is a submersion and hence for each $(p,q)\in M\times M$ the fibre

\begin{align*}
\Omega_{pq}=\{\gamma\in\Omega:\gamma(0)=p,\gamma(1)=q\}
\end{align*}
is a submanifold of codimension $2n$ whose tangent space at
$\gamma\in\Omega_{pq}$ is the subspace

\begin{align*}
H^1_0(\gamma)=\{\xi\in H^1(\gamma):\xi(0)=0,\,\xi(1)=0\}
\end{align*}
of $H^1(\gamma)$. Note that the family of spaces $H^1_0(\gamma)$,
$\gamma\in\Omega$, are the fibres of the vertical bundle $TF(\pi)$ of the submersion $\pi$.\\
We have for each semi-Riemannian metric $g$ on $M$ the associated energy
functional

\begin{align*}
f_g:\Omega\rightarrow\mathbb{R},\quad
f_g(\gamma)=\frac{1}{2}\int^1_0{g(\gamma',\gamma')\,dt}.
\end{align*}
The critical points of the restriction of $f_g$ to the submanifold $\Omega_{pq}$
of $\Omega$ are precisely the geodesics with respect to $g$ joining $p$ and
$q$.\\
Now let $g$ and $p,q\in M$ be fixed and let $\gamma\in\Omega_{pq}$ be a
geodesic. Then the second variation of $f_g$ at $\gamma$ is the quadratic form

\begin{align*}
h_\gamma(\xi)=\int^1_0{g\left(\frac{\nabla}{dt}\xi,\frac{\nabla}{dt}\xi\right)\,
dt}-\int^1_0{g(R(\gamma',\xi)\gamma',\xi)\,dt}
\end{align*} 
on the space $H^1_0(\gamma)$, where $\frac{\nabla}{dt}$ denotes the covariant
derivative along $\gamma$ and $R$ is the curvature of the Levi-Civita connection induced by $g$. If $g$ is
a Riemannian metric, then the Morse index of $\gamma$ is by definition the
dimension of the maximal subspace of $H^1_0(\gamma)$ on which $h_\gamma$ is
negative definite. For a non-Riemannian metric, there always exists an infinite
dimensional subspace on which $h_\gamma$ is negative definite, but it is shown
in \cite{Pejsachowicz} how one can define an index for geodesics in general
semi-Riemannian manifolds by using the spectral flow. The construction goes as
follows:\\
the geodesic $\gamma$ induces canonically a path $\Gamma$ in $\Omega$ by setting
$\Gamma(s)[t]=\gamma(s\cdot t)$, $s,t\in I$. Each $\Gamma(s)$ is a critical
point of the restriction of $f_g$ to $\Omega_{\gamma(0),\gamma(s)}$ and the
associated second variations $h_{\Gamma(s)}$,  $s\in I$, can be regarded as a map
on the pullback bundle $\Gamma^\ast(TF(\pi))$, which is a Hilbert bundle over
the unit interval $I$ having as typical fibre the Sobolev space
$H^1_0(I,\mathbb{R}^n)$. Now choose a global trivialisation $\psi:I\times
H^1_0(I,\mathbb{R}^n)\rightarrow\Gamma^\ast(TF(\pi))$ and obtain a path of
quadratic forms $h_{\Gamma(s)}(\psi(s,u))$, $(s,u)\in I\times
H^1_0(I,\mathbb{R}^n)$, whose Riesz representations $L_s$, $s\in I$, are shown
to be Fredholm in \cite[Prop. 3.1]{Pejsachowicz}. Moreover, it is easy to see
that $L_0$ is invertible. If $\ker h_\gamma=0$, then $L_{1}$ is invertible as
well and we call the geodesic $\gamma$ \textit{non-degenerate}. In this case we
obtain a path $L:I\rightarrow\Phi_S(H^1_0(I,\mathbb{R}^n))$ having invertible endpoints and we
define the \textit{spectral index} of the non-degenerate geodesic $\gamma$ by 

\begin{align*}
\mu(\gamma)=-\sfl(L)\in\mathbb{Z}.
\end{align*}
If $g$ is a Riemannian metric, then $\mu(\gamma)$ coincides with the Morse index
of $\gamma$, which can be computed as the total number of conjugate points along
$\gamma$ according to the classical Morse index theorem. Moreover, let us
mention that $\mu(\gamma)$ can also be obtained in general by counting
algebraically the number of conjugate points along $\gamma$, which is the
content of the semi-Riemannian Morse index theorem \cite{Pejsachowicz}.\\
Now let $X$ be a smooth manifold of dimension $n$ and let
$p:\mathcal{S}^{2}_\nu(M)\rightarrow M$ be the bundle of non-degenerate
symmetric two-forms of index $\nu$ on $M$. A \textit{family of semi-Riemannian
metrics} of index $\nu$ on $M$ is a smooth map $g:X\times
M\rightarrow\mathcal{S}^{2}_\nu(M)$ such that each map
$g(\lambda,\cdot):M\rightarrow\mathcal{S}^{2}_\nu(M)$, $\lambda\in X$, is a section of $p$. We
assume that there exists a smooth map $\sigma:X\times I\rightarrow M$ such that
each $\sigma(\lambda)=\sigma(\lambda,\cdot)\in\Omega$, $\lambda\in X$, is a
geodesic in $M$ with respect to the metric $g_\lambda$. We refer to $\sigma$ as
the trivial branch of geodesics.\\
We call $\lambda^\ast\in X$ a \textit{bifurcation point for geodesics} from the
trivial branch $\sigma$ if there exists a sequence
$(\lambda_n,\gamma_n)\rightarrow(\lambda^\ast,\sigma(\lambda^\ast))$ in
$X\times\Omega$ such that each $\gamma_n$ is a geodesic with respect to
$g_{\lambda_n}$,
\begin{align*}
\gamma_n(0)=\sigma(\lambda_n,0),\quad \gamma_n(1)=\sigma(\lambda_n,1)
\end{align*}    
and $\gamma_n\neq\sigma(\lambda_n)$, $n\in\mathbb{N}$.\footnote{%
We want to point out that the definition of a bifurcation point of geodesics is stated in an incorrect way in \cite{MussoPejsachowicz} and that all results proved in that paper actually concern bifurcation with respect to the definition given here.} 
 We denote the set of all
bifurcation points by $\mathcal{B}_\sigma\subset X$.\\
The main result of this section reads as follows:

\begin{theorem}\label{theoremIII}
Let $X$ be a smooth connected manifold of finite dimension $n$.
\begin{enumerate}
	\item[i)] Assume that there exist $\lambda_0,\lambda_1\in X$ such that
$\sigma(\lambda_0)$, $\sigma(\lambda_1)$ are non-degenerate and 
	
	\begin{align*}
	\mu(\sigma(\lambda_0))\neq\mu(\sigma(\lambda_1)).
	\end{align*}
Then:

	\begin{enumerate}
	\item[a)] Either $\mathcal{B}_\sigma$ has interior points or it disconnects some
open connected subset of $X$. 
	\item[b)] $\dim\mathcal{B}_\sigma\geq n-1$.
	\item[c)] If $X$ is simply connected, then $\mathcal{B}_\sigma$ disconnects $X$. 
\end{enumerate} 
\item[ii)] If $X$ is simply connected and there exists a sequence
$\{\lambda_k\}_{k\in\mathbb{N}}\subset X$ such that $\sigma(\lambda_k)$ is
non-degenerate for all $k\in\mathbb{N}$ and

\begin{align*}
\lim_{k \to +\infty}|\mu(\sigma(\lambda_k))|=+\infty,
\end{align*}
then $X\setminus\mathcal{B}_\sigma$ has infinitely many path components.
\end{enumerate}
\end{theorem}

\begin{proof}
We consider the endpoint map 

\begin{align*}
e:X\rightarrow M\times M,\quad e(\lambda)=(\sigma(\lambda,0),\sigma(\lambda,1)).
\end{align*}
The pullback $e^\ast(\pi)$ of the submersion $\pi$ can be defined in the usual way and its total space
is given by

\begin{align*}
E:=\{(\lambda,\gamma)\in X\times\Omega:e(\lambda)=\pi(\gamma)\}\subset
X\times\Omega.
\end{align*}
Note that
$E_\lambda=(e^\ast(\pi))^{-1}(\lambda)=\Omega_{\sigma(\lambda,0),\sigma(\lambda,
1)}$, $\lambda\in X$. By standard transversality arguments (cf. \cite[II,\S
2]{Lang}), $e^\ast(\pi)$ is a submersion as well and we obtain a commutative
diagram

\begin{align}\label{pullbackpi}
\xymatrix{
E\ar[r]^\iota\ar[d]_(0.42){e^\ast(\pi)}&\Omega\ar[d]^(0.45){\pi}\\
X\ar[r]^(0.37){e}&M\times M
}
\end{align}
Because of the commutativity of \eqref{pullbackpi}, the map
$\sigma:X\rightarrow\Omega$ induces a section $\tilde{\sigma}:X\rightarrow E$ of
$E$.\\
We now consider the smooth function

\begin{align*}
f:E\rightarrow\mathbb{R},\quad
f(\lambda,\gamma)=\frac{1}{2}\int^1_0{g_\lambda(\gamma',\gamma')\,dt}
\end{align*}
and note that $\gamma\in E_\lambda=\Omega_{\sigma(\lambda,0),\sigma(\lambda,1)}$
is a critical point of $f_\lambda:=f\mid_{E_\lambda}$ if and only if $\gamma$ is
a geodesic for the metric $g_\lambda$, $\lambda\in X$. Hence we have
reduced the study of $\mathcal{B}_\sigma$ to the bifurcation of critical points of the
functional $f$ from the branch $\tilde{\sigma}(X)$.\\
In order to apply theorems \ref{theoremI} and \ref{theoremII}, we need the
vector-bundle neighbourhood theorem, which is proved in \cite[App.
A]{MussoPejsachowicz} for bundles over compact base spaces. However, it is
easily seen by arguing as in \cite[Lemma 12.6]{Broecker} that its statement can be
proved in the non-compact case along the same lines. Accordingly, there exists a
trivial Hilbert bundle $X\times H$ over $X$ and a smooth map $\psi:X\times
H\rightarrow E$ such that $\psi$ is a diffeomorphism onto an open neighbourhood
of $\tilde{\sigma}(X)$ in $E$ and $\psi(\lambda,0)=\tilde{\sigma}(\lambda)$,
$\lambda\in X$. Now consider $\tilde{f}:X\times H\rightarrow\mathbb{R}$,
$\tilde{f}=f\circ\psi$. Since the restriction $\psi_\lambda$ to the fibre is a
diffeomorphism onto an open subset of $E_\lambda$, we infer that $u\in H$ is a
critical point of $\tilde{f}_\lambda$ if and only if $\psi_\lambda(u)$ is a
critical point of $f_\lambda$. In particular, $0\in H$ is a critical point of
all maps $\tilde{f}_\lambda$, $\lambda\in X$, and, moreover, $B(\tilde{f})=\mathcal{B}_\sigma$.\\
The Hessian $\tilde{L}_\lambda$ of $\tilde{f}_\lambda$ at $0\in H$ is the Riesz
representation of the quadratic form

\begin{align*}
\tilde{h}_\lambda(u)=h_\lambda((T_0\psi_\lambda) u),\quad u\in H,\,\,\lambda\in
X,
\end{align*} 
and $\tilde{L}:X\rightarrow\Phi_S(H)$ is a continuous family of selfadjoint
Fredholm operators.\\
We now choose any path $\gamma:I\rightarrow X$ such that $\gamma(0)=\lambda_0$
and $\gamma(1)=\lambda_1$ and we consider the composite map $\tilde{L}\circ\gamma:I\rightarrow\Phi_S(H)$, which is a path of selfadjoint Fredholm operators having invertible endpoints. Arguing as in \cite[Prop.
8.4]{MussoPejsachowicz}, we obtain

\begin{align*}
\sfl(\gamma,\tilde{L})=\sfl(\tilde{L}
\circ\gamma)=\mu(\sigma(\lambda_0))-\mu(\sigma(\lambda_1)).
\end{align*}
It is now clear that the theorems \ref{theoremI} and \ref{theoremII} apply to
the family $\tilde{f}:X\times H\rightarrow\mathbb{R}$. Since
$B(\tilde{f})=\mathcal{B}_\sigma$, we finally obtain the claim.  
\end{proof}

We finish this section by giving an example of theorem \ref{theoremIII}. Let $g$
be a semi-Riemannian metric on a smooth connected manifold $M$ of finite dimension $n$
such that $(M,g)$ is geodesically complete. We set $X:=TM$ and define as trivial branch of geodesics

\begin{align*}
\sigma:TM\times I\rightarrow M,\quad \sigma(v,t)=\exp_{\pi_M(v)}(t\cdot v),
\end{align*}
where $\pi_M:TM\rightarrow M$ denotes the projection of the tangent bundle. Here
we use that by assumption the exponential map is defined on the whole 
tangent bundle $TM$. Note that the family
$\sigma$ consists of all geodesics in $(M,g)$.\\
The constant paths $\sigma(0,t)\equiv p\in M$, $0\in T_pM\subset TM$ are
geodesics having vanishing spectral index $\mu(\sigma(0))=0$. Accordingly, if
there exists $v\in TM$ such that $\sigma(v)$ is non-degenerate and
$\mu(\sigma(v))\neq 0$, we infer that $\dim\mathcal{B}_\sigma\geq 2n-1$. Now let us
assume in addition that $\pi_1(M)=0$. Then $TM$ is simply connected and we
obtain that $\mathcal{B}_\sigma$ disconnects $TM$. Finally, if we can find a sequence
$\{v_k\}_{k\in\mathbb{N}}\subset TM$ such that $\sigma(v_k)$, $k\in\mathbb{N}$,
is non-degenerate and $|\mu(\sigma(v_k))|\rightarrow\infty$, 
$k\rightarrow\infty$, then $TM\setminus\mathcal{B}_\sigma$ has infinitely many path
components. Examples of the latter case are provided by the spheres $S^n$,
$n\geq 2$, with their usual Riemannian metrics.

\vspace{1cm}
Alessandro Portaluri\\
Department of Agriculture, Forest and Food Sciences\\
Universit\`a degli studi di Torino\\
Via Leonardo da Vinci, 44\\
10095 Grugliasco (TO)\\
Italy\\
E-mail: alessandro.portaluri@unito.it

\vspace{1cm}
Nils Waterstraat\\
Dipartimento di Scienze Matematiche\\
Politecnico di Torino\\
Corso Duca degli Abruzzi, 24\\
10129 Torino\\
Italy\\
E-mail: waterstraat@daad-alumni.de

\end{document}